\documentclass{amsart}
\usepackage[utf8]{inputenc}
\usepackage{amsmath, amssymb}
\usepackage{amsfonts}
\usepackage{graphicx}
\usepackage{tikz}
\usepackage{pict2e}
\usepackage{harpoon}
\usepackage{curves}
\usepackage{chngcntr}
\usepackage{apptools}
\usepackage{enumerate}
\usepackage{float}
\usepackage{verbatim}
\usepackage{hyperref}
\usepackage[nospace,noadjust]{cite}
\usetikzlibrary{calc,decorations.markings}
\AtAppendix{\counterwithin{theorem}{section}}
\numberwithin{equation}{section}

\newtheorem{theorem}{Theorem}[section]
\newtheorem{lemma}[theorem]{Lemma}

\newtheorem{proposition}[theorem]{Proposition}
\newtheorem{remark}[theorem]{Remark}

\newcommand{\R}{\mathbb{R}}

\newcommand{\N}{\mathbb{N}}
\newcommand{\C}{\mathbb{C}}

\newcommand{\cA}{{\mathcal A}}
\newcommand{\cB}{{\mathcal B}}
\newcommand{\cC}{{\mathcal C}}

\newcommand{\cG}{{\mathcal G}}

\newcommand{\cK}{{\mathcal K}}
\newcommand{\cL}{{\mathcal L}}

\newcommand{\cN}{{\mathcal N}}

\newcommand{\cP}{{\mathcal P}}

\newcommand{\cR}{{\mathcal R}}
\newcommand{\cS}{{\mathcal S}}

\newcommand{\beq}{\begin{equation}}
\newcommand{\eeq}{\end{equation}}
\newcommand{\bea}{\begin{eqnarray}}
\newcommand{\eea}{\end{eqnarray}}

\newcommand{\id}{\mathbb{I}}
\newcommand{\I}{\mathrm{i}}
\newcommand{\E}{\mathrm{e}}

\DeclareMathOperator{\Ai}{Ai}

\DeclareMathOperator{\imaginary}{Im}
\DeclareMathOperator{\dist}{dist}

\makeatletter
\newcommand{\dlmf}[1]{%
\cite[%
  \def\nextitem{\def\nextitem{, }}%
  \@for \el:=#1\do{\nextitem\href{http://dlmf.nist.gov/\el}{(\el)}}%
]{dlmf}%
}
\makeatother

\makeatletter
\@namedef{subjclassname@1991}{\emph{2020} Mathematics Subject Classification}
\makeatother

\author[M. Piorkowski]{Mateusz Piorkowski}
\address{Faculty of Mathematics\\ University of Vienna\\
Oskar-Morgenstern-Platz 1\\ 1090 Wien}
\email{\href{mailto:Mateusz.Piorkowski@univie.ac.at}{Mateusz.Piorkowski@univie.ac.at}}

\title[Riemann--Hilbert Theory without local Parametrix Problems]{Riemann--Hilbert Theory without local Parametrix Problems: Applications to Orthogonal Polynomials}

\date{}

\thanks{Research supported by the Austrian Science Fund (FWF) under Grant W1245.}
\begin{document}
\maketitle
\tableofcontents
\begin{abstract}
We study whether in the setting of the Deift--Zhou  nonlinear steepest descent method one can avoid solving local parametrix problems explicitly, while still obtaining asymptotic results. We show that this can be done, provided an a priori estimate for the exact solution of the Riemann--Hilbert problem is known. This enables us to derive asymptotic results for orthogonal polynomials on $[-1,1]$ with a new class of weight functions. In these cases, the weight functions are too badly behaved to allow a reformulation of a local parametrix problem to a global one with constant jump matrices. Possible implications for edge universality in random matrix theory are also discussed.
\end{abstract}

\section{Introduction}
\subsection{Background}
Local parametrix problems appear frequently in the context of the nonlinear steepest descent method for Riemann--Hilbert (R-H) problems, formulated by Deift and Zhou (\cite[Sect.~4]{DIZ}, \cite[Sect.~7]{DKMVZ2}, \cite{DZ}, \cite[Sect.~4]{DZPainleve}, for details see \cite{DZNLS}). Solutions to parametrix problems can determine either the leading asymptotics (\cite{DZNLS}, \cite{GT}, \cite{KT}), or contribute to higher-order corrections (\cite[Ch.~7]{PD}, \cite{DIZ}, \cite{DKMVZ2}, \cite{KMVV}). Interestingly, even in the second case one has to construct an explicit local parametrix solution to obtain rigorous leading asymptotics. Hence, the natural question arises, whether such a construction can be avoided. A positive answer would be useful in applications for local parametrix problems with no known solutions (see \cite{DC}, \cite[Sect.~5]{DKMVZ3}). 

In this paper we show how an explicit construction of a local parametrix solution can be avoided and use our method to obtain new error estimates for Plancherel--Rotach asymptotics of orthogonal polynomials. Asymptotics of orthogonal polynomials were studied thoroughly by Bernstein and Szeg\H{o} on the unit interval and unit circle (see \cite[Ch.~12]{GS} and references in therein). There has been renewed interested in these asymptotics motivated by the Wigner--Dyson--Mehta universality conjecture in random matrix theory (\cite{Dyson1962}, \cite{Dyson1970}, \cite{Mehta71}, \cite{Mehta}). In this setting orthogonal polynomials can be applied most naturally to unitary ensembles (\cite{BI}, \cite{PD}, \cite{DKMVZ3}, \cite{PS97}), but also to orthogonal and symplectic ensembles (\cite{DGBulk}, \cite{DGEdge}, \cite{SNW}, \cite{NW1}). The R-H formulation, first introduced by Fokas, Its and Kitaev (\cite{FIK2}, \cite{FIK1}), in conjunction with the nonlinear steepest descent method is particularly useful in this context. The nonlinear steepest descent method was applied to orthogonal polynomials on the real line by Bleher and Its in \cite{BI}, (see also \cite{DIZ}) and by Deift, Kriecherbauer, McLaughlin, Venakides and Zhou in \cite{DKMVZ3} and \cite{DKMVZ2} (for an introduction see the book by Deift \cite{PD}). Based on this work Kuijlaar, McLaughlin, Van Assche and Vanlessen computed in \cite{KMVV} the asymptotics of  polynomials orthogonal on $[-1,1]$ and related quantities. The leading asymptotic terms were already known \cite[Thm.~12.1.1--4]{GS}. However, the R-H analysis leads to more explicit error terms and in the case of \cite{KMVV} even an asymptotic expansion of the orthogonal polynomials was obtained. The follow-up paper \cite{KV} relates these results to bulk and edge universality in random matrix theory (see also \cite{SNW}). We comment more on this topic in relation to our results in the discussion section. 

As part of the R-H analysis performed in \cite{KMVV} one is confronted with local parametrix problems at $x = \pm1$. The solvability of these problems puts constraints on the local behaviour of the weight function at the corresponding endpoints. In particular, the authors considered the modified Jacobi weight function $\rho^{\alpha,\beta}_{Jac}$:
\begin{align} \label{Jacobi}
    \rho^{\alpha,\beta}_{Jac}(x) := (1-x)^\alpha(1+x)^\beta h(x), \quad x \in (-1,1), \ \alpha, \beta > -1,
\end{align}
where $h$ is strictly positive on $[-1,1]$ and assumed to have an analytic continuation to a neighbourhood of $[-1,1]$ (note the inclusion of the endpoints $\pm 1$). In our work we do not consider the prefactor $(1-x)^\alpha(1+x)^\beta$, but rather assume that our weight function $\rho$ has an analytic continuation only in a lense-shaped neighbourhood of $(-1,1)$, together with some growth conditions near $x = \pm 1$. The difference might seem to be minor, but the possibility that the weight function does not have an analytic continuation in a neighbourhood of $x = \pm 1$ makes the usual R-H analysis impossible (see \cite[Sect.~6]{KMVV}). In particular, the reformulation of the local parametrix problem to a global one with constant jump matrices relies on the local analyticity of the function $h$ in \eqref{Jacobi} around $x = \pm 1$.

As already mentioned, the methods described in this papers can be used for problems that do not have known parametrix solutions. This was the case in \cite[Sect.~5]{DKMVZ3}, where the authors had to rely on Fredholm index theory of singular integral operators (cf.~\cite{XZ}) and prove an additional uniqueness result.\footnote{See \cite{ParametrixFinal} for a similar application of Fredholm index theory to the R-H problem for the KdV equation.} Another recent example can be found in \cite{DC}, where the weight function $\rho_{\log}$ on $[-1,1]$ with a logarithmic singularity at $x = 1$ was considered:
\begin{align}
    \rho_{\log}(x) := \log \dfrac{2k}{1-x}, \quad x \in(-1,1), \ k > 1. 
\end{align}
While a parametrix solution has not been found, the authors managed to circumvent this issue through a comparison argument with the Legendre problem $(\rho_{Leg}(x) \equiv 1$). However, the analytic continuation of $\rho_{\log}$ around the point $x = 1$, which introduces an explicit jump condition on $(1, 1+\delta), \ \delta > 0$, has been crucial. The weights considered in this paper, are not required to have such analytic continuation around the endpoints. 

\subsection{Outline of this paper}
In the next section we set the stage by discussing the notion of approximating solutions to R-H problems. We then show that the construction of a local parametrix solution can be avoided, provided a certain a priori $L^p$-estimate of the exact solution to the global R-H problem and a regular enough model solution is known. Our method uses the connection between R-H problems and singular integral equations, which will be briefly summarized. 

In Section 3 we describe parametrix problems as they appear in practise, and analyse them using our method. We summarize our findings in Theorem \ref{mainTheorem} and also state Lemma \ref{extension} which will be crucial for obtaining the a priori $L^p$-estimate. 

 Following this, an application of our approach is presented for the case of orthogonal polynomials on the interval $[-1,1]$. We consider a new class of weight functions and obtain a bound of the error term in the pointwise asymptotics of the orthogonal polynomials, as the degree goes to infinity. We then elaborate on why the reformulation of the local parametrix problem to a global one with constant jump matrices, as performed in \cite[Sect.~6]{KMVV}, is not possible. Finally, we illustrate how our method uses the exact R-H solution as a local parametrix solution.

 In the discussion section we elaborate on connections with random matrix theory, in particular eigenvalue universality near the edge of the spectrum. Moreover, we mention the advantages and limitations of our approach and point towards future challenges related to obtaining the a priori $L^p$-estimate. 
 
 The first appendix contains a description of the Airy parametrix problem, including the heuristics by which the explicit parametrix solution can be found. In the second appendix it is shown that a local a priori $L^p$-estimate, instead of the one described in Section 3, is also sufficient. This is crucial in applications different than the one considered in this paper.
 
\section{Approximating solutions of R-H problems}
\subsection{Two R-H problems}
Consider a R-H problem with data $(v_\cS, \Sigma)$, meaning with jump matrix $v_\cS$ and jump contour $\Sigma$. We are looking for a matrix-valued function $S$, normalized at infinity, such that:
\vspace{10pt}
\begin{enumerate}[(i)]
\item $S(z)$ is analytic for $z \in \C \setminus \Sigma$,
\\
\item $S_+(k) = S_-(k)v_\cS(k)$, \ for $k \in \Sigma$,
\\
\item $S(z) = \id + O(z^{-1})$, \ as  $z \rightarrow \infty$.
\end{enumerate}
\vspace{10pt}
Note that condition (iii) is not specified by the data $(v_\cS, \Sigma)$ and has to be stated separately. Here $\Sigma$ is a 'sufficiently smooth' oriented contour and the $+ (-)$ sign corresponds to taking the left (right) limit to $\Sigma$. The precise conditions on $\Sigma$ and the sense in which the limits are taken can be found \cite{JL2}.

Next, consider a model R-H problem with data $(v_\cN, \Sigma^{mod})$ with  $\Sigma^{mod} \subseteq \Sigma$, and the same normalization at infinity as $S$. Hence we are looking for a matrix valued-function $N$, such that:
\vspace{10pt}
\begin{enumerate}[(i)] 
\item $N(z)$ is analytic for $z \in \C \setminus \Sigma^{mod},$
\\
\item $N_+(k) = N_-(k)v_\cN(k)$, \ for $k \in \Sigma^{mod},$
\\
\item $N(z) = \id + O(z^{-1})$, \ as $z \rightarrow \infty.$
\end{enumerate}
\vspace{10pt}
We refer to $S$ as the \emph{exact solution} and to $N$ as the \emph{model solution}.\footnote{Sometimes the term \emph{global parametrix solution} is used for $N$, which can be mistaken to be a solution to the local parametrix problem in its limiting global form, so we refrain from using it.}
The general question we try to answer in this section is the following:
\begin{center}
\begin{align} \label{question1}
\begin{minipage}{23em}
  Under what conditions is the model solution $N$ a good approximation to the exact solution $S$?
\end{minipage}
\end{align}
\end{center}
\vspace{10pt}
Before tackling \eqref{question1}, we need to answer the more basic question:
\begin{center}
\begin{align} \label{question2}
\begin{minipage}{23em}
  What does it mean for a model solution $N$ to be a good approximation to the exact solution to $S$?
\end{minipage}
\end{align}
\end{center}
\vspace{10pt}
The answer to question \eqref{question2} depends on the problem at hand. In the case of orthogonal polynomials pointwise estimates are of interest, meaning we would like $S(z)$ to be close to $N(z)$ for $z \in \C \setminus \Sigma$. In problems involving scattering theory we are interested in the first Laurent term of $S$ at infinity, meaning the complex number $\theta_\cS$ given by  
\beq
S(z) = \id + \dfrac{\theta_\cS}{z} + O(z^{-2}), \hspace{7pt} z \rightarrow \infty.
\eeq
Hence, we would like $\theta_\cN$ with
\beq
N(z) = \id + \dfrac{\theta_\cN}{z} + O(z^{-2}), \hspace{7pt} z \rightarrow \infty,
\eeq
to be close to $\theta_\cS$. 

Usually, the jump matrix $v_\cS$ (and sometimes $v_\cN$)\footnote{In the case of orthogonal polynomials $v_\cN$ is independent of the degree $n$. For integrable wave equations in the elliptic wave region, $v_\cN$ is periodic in the time parameter (see \cite{EPT}).} depends on some auxiliary continuous or discrete parameter. In the case of orthogonal polynomials this parameter is the polynomial degree $n$ and we demand for the $n$-dependent solution $S$:
\beq
S(z,n) = (\id + o(1))N(z), \hspace{7pt} z \in \C \setminus \Sigma, \quad n \rightarrow \infty.
\eeq
In the case of scattering theory this parameter, denoted by $t$, is time and we demand:
\beq
\theta_\cS(t) = \theta_\cN(t) + o(1), \hspace{7pt} t \rightarrow \infty,
\eeq
where the error term is a measure of the accuracy of the approximation. Having answered question \eqref{question2}, we now move to question \eqref{question1}. For this we need to reformulate a R-H problem as a singular integral equation. 

\subsection{Singular integral formulation of R-H problems}
 To find approximations to solutions of R-H problems as described in the last section, we need to reformulate a R-H problem as an equivalent singular integral equation. The underlying theory can be found in \cite{BC}, \cite{MUSK}, \cite{XZ}, for more recent developments see  \cite{JL2}, \cite{ParametrixFinal}. Let us define the Cauchy operator $\mathcal C^{\Sigma}$ associated to an oriented contour $\Sigma$:
\beq
\mathcal{C}^\Sigma: L^p(\Sigma) \rightarrow \mathcal{O}(\C \setminus \Sigma), \hspace{7pt} f \rightarrow \mathcal{C}^\Sigma(f)(z) := \dfrac{1}{2\pi \I} \int_\Sigma \dfrac{f(k)}{k-z} \hspace{3pt}dk, 
\eeq
with $p \in (1,\infty)$, which shall be assumed throughout this section. The only further requirement needed for $\mathcal C^\Sigma$ to be well-defined is that $(k-z)^{-1}$ is in $L^q(\Sigma)$ with $p^{-1} + q^{-1} = 1$ for some, and hence for every $z \in \C \setminus \Sigma$. Given some further regularity assumptions on $\Sigma$ which are fulfilled in most applications including ours \cite{JL2}, we can define two bounded operators given by:
\beq
\mathcal{C}^\Sigma_\pm: L^p(\Sigma) \rightarrow L^p(\Sigma), \hspace{7pt} f \mapsto \mathcal{C}^\Sigma_\pm(f)(k) := \lim\limits_{z \rightarrow k \pm} \mathcal{C}^\Sigma(f)(z),
\eeq
where the limits are assumed to be nontangential, in which case they exist a.e. on $\Sigma$. 

We now turn to a bijection between solutions of R-H problems and solutions of a certain singular integral equation. These results can be found in \cite{XZ}. We assume that $w_{\mathcal R} := v_{\mathcal R} - \id \in L^p(\Sigma)$\footnote{Usually it is assumed that $w_ \mathcal R \in L^\infty(\Sigma)$, which would imply $M_{w_\mathcal R}^\Sigma = L^p(\Sigma)$ and the boundedness of $\mathcal C_{w_\mathcal R}^\Sigma$. This assumption will not be needed and even violated in our application to orthogonal polynomials}, where we abuse notation by denoting as $L^p(\Sigma)$ the space of matrix functions with entries in $L^p(\Sigma)$ (the subscript $\mathcal R$ is choosen for later convenience). Associated to $w_\mathcal R$ and $\Sigma$, let $M_{w_\mathcal R}^\Sigma$ be the maximal domain of the multiplication operator defined by $w_\mathcal R$, meaning
\beq
M_{w_\mathcal R}^\Sigma := \lbrace f \in L^p(\Sigma) : fw_\mathcal R \in L^p(\Sigma) \rbrace.
\eeq
With this we can define the operator $\mathcal{C}_{w_\mathcal R}^\Sigma$ associated to a R-H problem with data $(v_\mathcal R,\Sigma)$:
\begin{align}
\begin{split}
\mathcal{C}_{w_\mathcal R}^{\Sigma}& \colon M_{w_\mathcal R}^\Sigma \rightarrow L^p(\Sigma), \hspace{7pt} f \mapsto \mathcal{C}^{\Sigma}_-(f w_\mathcal R). 
\end{split}
\end{align}
The operator $\mathcal{C}^\Sigma_-$ is evaluated componentwise for matrix inputs. To next proposition found in \cite[Prop. 3.3]{XZ} describes the correspondence between R-H problems and certain singular integral equation and is central for our approach.  
\begin{proposition}
Let $(v_\mathcal R,\Sigma)$ be the data of a R-H problem, and assume that $w_\mathcal R := v_\mathcal R-\id \in L^p(\Sigma)$. Then there is a bijection between R-H solutions $R$, satisfying
\beq
\lim_{z \rightarrow \infty} R(z) \rightarrow \id, \quad R_\pm - \id \in L^p(\Sigma)
\eeq
and solutions $\Phi \in M_{w_\mathcal R}^\Sigma$ of \beq \label{SIE}
(\id - \mathcal C_{w_\mathcal R}^\Sigma)\Phi = \mathcal C_-^\Sigma(w_\mathcal R).
\eeq
Moreover the relation between $R$ and $\Phi$ is given by
\begin{align}\label{RfromPhi}
R &= \id + \mathcal C^\Sigma((\Phi+\id)w_\cR),
\\\label{PhifromR}
\Phi &= R_- -\id.
\end{align}
\end{proposition}
\begin{proof}
Let $R$ be a solution of the R-H problem and define $\Phi := R_- - \id$. Then by assumption $\Phi \in L^p(\Sigma)$ and as
\beq
\Phi w_\mathcal R = (R_--\id)(v_\mathcal R-\id) = R_+ - R_- - w_\mathcal R \in L^p(\Sigma)
\eeq
we indeed see that $\Phi \in M_\mathcal R^\Sigma$.
Next, using the Sokhotski--Plemelj formula for additive R-H problems, we obtain from
\beq
R_+-R_- = R_- w_\mathcal R
\eeq
the equality
\beq \label{defS}
R = \id + \mathcal C^\Sigma(R_- w_\mathcal R).
\eeq
Taking the limit to the contour $\Sigma$ from the right, we get
\beq
R_- = \id + \mathcal C^\Sigma_{w_\cR}(R_-).
\eeq
which after substituting $R_- = \Phi+\id$ is equivalent to \eqref{SIE}. 

Next, let $\Phi \in M_{w_\mathcal R}^\Sigma$ satisfy \eqref{SIE} and define $R := \id + \mathcal C^\Sigma((\Phi+\id)w_\mathcal R)$. Note that from the assumptions on $\Phi$ and $w_\mathcal R$ it follows that $R(z) \rightarrow \id$, as $z \rightarrow \infty$. By the Sokhotski-Plemelj formula we have
\beq \label{R+R-}
R_+ - R_- = (\Phi+\id)w_\mathcal R.
\eeq
On the other hand we compute
\begin{align}
\begin{split}
R_- - \id &= \mathcal C^\Sigma_-((\Phi+\id)w_\mathcal R) 
\\
&= \mathcal C^\Sigma_{w_\mathcal R}(\Phi) + C^\Sigma_-(w_\mathcal R) 
\\
&= \Phi,
\end{split}
\end{align}
as $\Phi$ satisfies \eqref{SIE}. Substituting this into \eqref{R+R-} results in 
\beq
R_+ = R_- v_\mathcal R.
\eeq
Furthermore, we have
\begin{align}
\begin{split}
R_- - \id &= \Phi \in L^p(\Sigma)
\\
R_+ - \id &= \Phi w_\mathcal R + \Phi + w_\mathcal R \in L^   p(\Sigma).
\end{split}
\end{align}
Hence we see that $R$ is a solution of the R-H problem with the required properties and the proof is finished. 
\end{proof}
\subsection{Residual R-H problem}
Let us now return to the solutions $S$ and $N$ from the beginning of the section, under the assumption that $v_\cS$ and hence $S$ depends on a discrete parameter $n \in \mathbb{N}_0$. Moreover let $N$ and $N_\pm$ be invertible. As $v_\mathcal N = N_+ N_-^{-1}$, this implies that $v_\cN$ is also invertible.

We can now define a new matrix-valued function $R := S N^{-1}$. Assumings $\Sigma^{mod} \subseteq \Sigma$, we see that $R$ will have jumps only on $\Sigma$. We call $R$ the \emph{residual solution} and it satisfies the residual R-H problem with data $(v_\cR, \Sigma)$, where 
\beq
v_\cR := N_- v_\cS N_+^{-1}.
\eeq
Assuming that $R_\pm-\id \in L^p(\Sigma)$, we know from \eqref{RfromPhi} that it can be also written in integral form: 
\begin{align} 
R(z,n) &= \id + \mathcal C^\Sigma((R_- w_\cR)(z,n) 
\notag\\\label{R}
&= \id + \dfrac{1}{2\pi \I} \int_\Sigma \dfrac{R_-(k,n)w_\cR(k,n)}{k-z}\hspace{3pt} dk
\\
\nonumber
&= \id + \dfrac{1}{2\pi \I} \int_\Sigma \dfrac{S_-(k,n)(v_\cS(k,n)-v_\cN(k))N_+^{-1}(k)}{k-z}\hspace{3pt} dk,
\end{align}
with $w_{\cR} = v_{\cR} - \id$. The quantity of interest is the $n$-dependent $L^1(\Sigma)$-norm of the integrand without $(k-z)^{-1}$:
\beq \label{integrand}
\big \Vert S_- (v_\cS-v_\cN) N_+^{-1} \big \Vert_{L^1(\Sigma)}.
\eeq
Let us assume that \eqref{integrand} is of order $\varepsilon(n)$ where $\varepsilon: \mathbb{N}_0 \rightarrow \R_+$. Observe that in this case we have
\beq
R(z,n) = \id + O(\varepsilon(n)\dist(z,\Sigma)^{-1}),
\eeq
which implies
\beq
S(z,n) = (\id +O(\varepsilon(n)\dist(z,\Sigma)^{-1})N(z).
\eeq
Here $\dist(z,\Sigma)$ denotes the distance between $z$ and $\Sigma$. Hence, we see that in order to show convergence of $N$ to $S$ away from the contour $\Sigma$, we need to control \eqref{integrand}.

\section{Riemann--Hilbert problems in applications}
We now compare the setting described in the last section with R-H problems appearing in practise. In the case of orthogonal polynomials
(\cite{DKMVZ2},\cite{KMVV})  and scattering theory (\cite{RAREFACTION},\cite{ParametrixFinal}), the contour $\Sigma$ is a union of $\Sigma^{mod}$ and $\Sigma^{exp}$. The $n$-independent matrix $v_\cN$ is given by
\beq \label{vN}
v_\cN(k) =
\begin{cases}
v_\cS(k,n), &k \in \Sigma^{mod},
\\
\id, &k \in \Sigma^{exp} \setminus \Sigma^{mod}.
\end{cases}
\eeq
\\
 We see that $v_\cS(k,n)$ is therefore also $n$-independent for $k \in \Sigma^{mod}$. Moreover, let us assume that $\det v_\cS \equiv \det v_\cN \equiv 1$. This condition implies by Liouville's theorem that both solutions $S$ and $N$ are unique and $\det S \equiv \det N \equiv 1$ \cite{PD}.

 On $\Sigma^{exp}$ the jump matrix $v_\cS$ converges uniformly exponentially fast as to the identity matrix as $n \rightarrow \infty$, except in the vicinity of a finite number of points $\kappa \in \mathcal{K} \subset \Sigma^{exp}$. For each $\kappa \in \mathcal{K}$ the local behaviour of $w_\cS = v_\cS-\id$ is given by:
\beq \label{vkappa}
|w_\cS(k,n)| =O(\E^{-cn|k-\kappa |^\chi}),
\eeq
for two positive constants $c$ and $\chi$. Moreover, $v_\cN$ is uniformly bounded for $z \in \C \setminus \Sigma^{mod}$. The same holds true for the $n$-independent model solution $N$, except in the vicinities of the points $\kappa \in \mathcal{K}$, where $N$ it can have fourth-root singularities\footnote{In our application to orthogonal polynomials in the next section, this condition can be violated for certain singular weight functions.} 
\beq \label{Nkappa}
| N(z) | = O(| z - \kappa |^{-1/4}).
\eeq
The same is true for $N^{-1}$ as $\det N \equiv 1$. 

In applications a local parametrix problem has to be solved in vicinities of the points $\kappa \in \cK$. The value of $\chi$ is critical, as it determines the class of special functions from which the explicit parametrix solution can be constructed. For $\chi = 2$, these are the parabolic cylindrical functions, and this case occurs in the study of nonlinear integrable systems in the dispersive region (\cite{DZ}, \cite{DZNLS}, \cite{KamvissisToda}, \cite{KT}, \cite{JL1}). The cases $\chi = 1/2$ and $\chi = 3/2$ are common in the R-H analysis of orthogonal  polynomials. Measures with finite support, lead to the Bessel parametrix problem at the endpoints \cite{KMVV}, while exponential measures on on the real line lead to the Airy parametrix problem (\cite{BI}, \cite[Ch.~5]{PD}, \cite{DKMVZ3}, \cite{DKMVZ2},  \cite{ParametrixFinal}). Moreover, $\chi = 3/2$ is also related to the Painl\'eve II equation (\cite{DZPainleve}, \cite{BaikDeift}) and appears in the analysis of integrable systems with rarefaction/steplike initial data (\cite{RAREFACTION}, \cite{KotlyarovmKdV}, \cite{ParametrixFinal}). 

Let us now apply H{\"o}lder's inequality to
\eqref{integrand} assuming condition \eqref{vN}:
\begin{align} 
\big \Vert S_-(v_\cS-v_\cN)N_+^{-1} \big \Vert_{L^1(\Sigma)}
&=
\big\Vert S_-w_\cS N_-^{-1} \big \Vert_{L^1(\Sigma^{exp})}
\notag\\\label{CS}
&\leq \Vert S_- \Vert_{L^p(\Sigma^{exp})} 
\big \Vert w_\cS N_-^{-1} \big \Vert_{L^q(\Sigma^{exp})}.
\end{align}
Note that we made use of $v_\cN(k) = \id$ for $k\in \Sigma^{exp}$, and as $v_\cN(k) = v_\cS(k,n)$ for $k \in \Sigma^{mod}$, we only need to integrate over $\Sigma^{exp}$. Now, from the assumptions \eqref{vkappa} and \eqref{Nkappa}, it follows that
\beq \label{estimateN}
 \Vert k^i w_\cS(k,n) N_-^{-1}(k)  \Vert_{L^q(\Sigma^{exp})} = O(n^{\frac{1}{4\chi}-\frac{1}{q\chi }}), \quad q \in [1,4),\ i \in \mathbb{N}_0,
\eeq
where the main contributions come from the points $\kappa \in \mathcal{K}$. The motivation for including $k^i$ will become clear in the next theorem. Note that the condition on $q$ implies $p \in (4/3, \infty]$. 

We see that in order to guarantee that \eqref{CS} goes to $0$, it is sufficient to show that
\beq \label{apriori}
\Vert S_- \Vert_{L^p(\Sigma^{exp})} = O(n^r)
\eeq
with 
\beq
r < \frac{1}{q\chi }-\frac{1}{4\chi}.
\eeq
We call estimates of the form \eqref{apriori} a priori $L^p$-estimates, as they have to be established before an approximation for $S$ is known.\footnote{A priori $L^p$-estimates of solutions have been considered in R-H theory previously by Deift and Zhou \cite{DZaPriori} in their study of long-time asymptotics of solutions to the perturbed nonlinear Schr\"odinger equation on the real line.} We show in the next section that such estimates can be computed in the case of the R-H problem associated with orthogonal polynomials on the interval $[-1,1]$. 

We can now state the following theorem:
\begin{theorem} \label{mainTheorem}
Suppose $v_\cS, v_\cN, S, N$ satisfy \eqref{vN}, \eqref{vkappa}, \eqref{Nkappa}, \eqref{apriori} and $\Sigma = \Sigma^{mod}\cup \Sigma^{exp}$. Let 
\beq
s := \frac{1}{q\chi }-\frac{1}{4\chi}-r > 0, \quad l \in \mathbb{N}_0.
\eeq
Then
\beq \label{Sasymptotics1}
S(z,n) = (\id + O(n^{-s}dist(z,\Sigma^{exp})^{-1}))N(z).
\eeq
for $z \in \C \setminus \Sigma^{exp}$. Moreover,
\beq \label{Sasymptotics2}
S(z,n) = \Big(\id+\sum_{i = 1}^\ell \dfrac{\theta_i(n)}{z^i}+ O(z^{-\ell-1})\Big)N(z),
\eeq
for $z \rightarrow \infty$ such that $|1-k/z|\geq \varepsilon > 0$ for all $k \in \Sigma^{exp}$, with the matrices $\theta_i$ satisfying
\beq
| \theta_i(n) |_\infty = O(n^{-s}).
\eeq
\end{theorem}
\vspace{5pt}
\begin{proof}
Equation \eqref{Sasymptotics1} follows from writing $S = RN$ and using the expression \eqref{R} for $R$, together with H\"{o}lder's inequality \eqref{CS} and the estimates \eqref{estimateN} and \eqref{apriori}. The asymptotic expansion follows analogously after substituting the partial Neumann series
\beq
\dfrac{1}{k-z} = -\sum_{i=1}^\ell\dfrac{k^{i-1}}{z^i}-\dfrac{k^{\ell}}{z^{\ell+1}}\dfrac{1}{1-k/z}.
\eeq
into the integrand in \eqref{R}.
\end{proof}
\begin{remark}
Instead of the a priori $L^p$-estimate \eqref{apriori}, local $L^p$-estimates around the points $\kappa \in \cK$ together with some auxiliary assumptions are sufficient for the conclusion of Theorem \ref{mainTheorem} to hold. For a proof and additional references, see Appendix B. 
\end{remark}
The next lemma tells us that a priori $L^p$-estimates of solutions to R-H problems can be uniformly extended to larger contours which are often introduced in the nonlinear steepest descent method. 
\begin{lemma} \label{extension}
Let $\Sigma$ and $\Gamma$ be oriented contours in $\hat \C := \C \cup \lbrace \infty \rbrace$. Assume that the Cauchy boundary operators $\mathcal{C}^{\Sigma \cup \Gamma}_\pm$ are well-defined and bounded on $L^p(\Sigma \cup \Gamma)$. Let $f \in \mathcal{O}(\hat \C \setminus \Sigma)$ be given, such that $f$ is continuous at infinity in the case that $\infty \in \Sigma$ and the $\pm$-limits on $\Sigma$ exist in the usual sense satisfying $f_+ - f_- \in L^p(\Sigma)$. Then 
\beq
\Vert f-f(\infty) \Vert_{L^p(\Gamma)} \leq C(\Gamma \cup \Sigma)\Vert f_+ - f_-\Vert_{L^p(\Sigma)} 
\eeq
for some positive constant $C(\Gamma  \cup \Sigma)$ independent of $f$.
\end{lemma}
\begin{proof}
Note that because of the conditions on $f$ it follows from the properties of the Cauchy integral operator that 
\beq
f - f(\infty)=  \mathcal{C}^{\Sigma}(f_+ - f_-) = \mathcal{C}^{\Sigma \cup \Gamma}(f_+ - f_-).
\eeq
where the last equality is true because $f_+ = f_-$ on $\Gamma \setminus \Sigma$, as $f \in \mathcal{O}(\hat \C \setminus \Sigma)$. Hence we conclude
\beq
\Vert f - f(\infty) \Vert_{L^p(\Gamma)} \leq \Vert \mathcal{C}^{\Sigma \cup \Gamma}_\pm \Vert_{L^p(\Sigma \cup \Gamma)} \Vert f_+ - f_- \Vert_{L^p(\Sigma \cup \Gamma)} = \Vert \mathcal{C}^{\Sigma \cup \Gamma}_\pm\Vert_{L^p(\Sigma \cup \Gamma)}\Vert f_+ - f_- \Vert_{L^p(\Sigma)}
\eeq
which shows that we can choose $C(\Sigma \cup \Gamma) = \Vert \mathcal{C}^{\Sigma \cup \Gamma}_\pm\Vert_{L^p(\Sigma \cup \Gamma)}$.
\end{proof}
\begin{remark}
Similar arguments work in the case that the R-H problem is not stated for a matrix $S$, but rather for a vector $s$ which is normalized to be $s_\infty$ at infinity. However, we still need a matrix-valued model solution $N$, which is normalized to the identity matrix at infinity. As before we define a vector-valued function $r := s N^{-1}$ which can be written in integral form
\beq
r(z,n) = s_\infty + \dfrac{1}{2\pi \I} \int_\Sigma \dfrac{s_-(k,n)(v_\cS(k,t)-v_\cN(k))N_+^{-1}(k,n)}{k-z} \hspace{3pt} dk.
\eeq
The rest of the analysis is analogous. An example of a vector-valued R-H problem comes from the inverse scattering transform of the KdV equation (\emph{\cite{EPT}, \cite{GT} \cite{ParametrixFinal}}).
\end{remark}
\section{Application to Orthogonal Polynomials on \texorpdfstring{$[-1,1]$}{}}
\subsection{Riemann--Hilbert formulation of orthogonal polynomials}
Let us consider an example for which our method can provide new results, namely the R-H problem for orthogonal polynomials on $[-1,1]$. We assume that the corresponding measure $d\mu$ on $[-1,1]$ is absolutely continuous and thus can be written as
\beq
d\mu(x) = \rho(x) \hspace{1pt} dx, \quad x \in (-1,1),
\eeq 
for some real-valued function $\rho \geq 0$. 
Following \cite{FIK2}, \cite{FIK1} the R-H problem characterizing the $n$-th orthogonal polynomial is stated as follows:
\\
\\
For any $n \in \N_0$ find a $2 \times 2$ matrix-valued function $X$ on $\C \setminus [-1,1]$, such that:
\vspace{10pt}
\begin{enumerate}[(i)]
\item $X(z,n)$ is analytic in for $z \in \C \setminus [-1,1]$,
\\
\item $X_+(x,n) = X_-(x,n) \begin{pmatrix}
1  & \rho(x)
\\
0 & 1
\end{pmatrix}, \quad$ for $x \in (-1,1),$
\\
\item $X(z,n) = (\id + O(z^{-1}))\begin{pmatrix}
(2z)^n & 0
\\
0 & (2z)^{-n}
\end{pmatrix}$, \quad as $z \rightarrow \infty.$
\end{enumerate}
\vspace{10pt}
A Liouville type argument shows that if a solution exists with square integrable singularities at $x = \pm 1$, it is necessary unique \cite{PD}. The following theorem explains how the R-H solution $X$ is related to orthogonal polynomials:
\begin{theorem}{(Fokas, Its, Kitaev (\emph{\cite{FIK2},\cite{FIK1}}))} \label{orthogonalMatrix}
The unique solution with square integrable singularities at $x = \pm 1$ of the R-H problem for orthogonal polynomials is given by
\beq \label{X}
X(z,n) = 
\begin{pmatrix}
p_n(z) & \mathcal{C}^{(-1,1)}(p_n \rho)(z)
\\
\eta_{n-1} p_{n-1}(z) & \eta_{n-1} \mathcal{C}^{(-1,1)}(p_{n-1} \rho)(z)
\end{pmatrix}
\eeq
where $p_n(z)$ is the $n$-th orthogonal polynomial with leading coefficient $2^n$ and
\beq
\eta_n := -\pi \I\Vert p_n \Vert^{-2}_{L^2((-1,1), \rho(x) dx)}.
\eeq
\end{theorem}
Usually the R-H problem for orthogonal polynomials is normalized without the factor $2^{\pm n}$ at infinity, in which case the $p_n(z)$ would be the monic orthogonal polynomials. The next theorem found in \cite[Thm.~12.7.1]{GS} explains this discrepancy:
\begin{theorem} \label{szego}
Assume the weight $\rho(x)$ on $(-1,1)$ satisfies the Szeg\H{o} condition, given by
\beq
\int_{-1}^1 \dfrac{\log \rho(x)}{\sqrt{1-x^2}} \hspace{3pt} dx > -\infty.
\eeq 
Then
\beq
\lim\limits_{n \rightarrow \infty} \Vert p_n \Vert_{L^2([-1,1], \rho(x) dx)} = \sqrt{\pi} \exp \Big( \dfrac{1}{2\pi} \int_{-1}^1 \dfrac{\log \rho(x)}{\sqrt{1-x^2}} \hspace{3pt} dx \Big).
\eeq
\vspace{1pt}
\end{theorem}
Hence, we see that with our normalization, the $L^2((-1,1), \rho(x)dx)$-norm of $p_n$ converges as $n$ goes to infinity. This uniform boundedness of the norm is similar, though not identical, to the a priori $L^p$-estimate needed for our approach to local parametrix problems. The difference is the measure $\rho(x) dx$, as we need a priori $L^p$-estimates in the space $L^p((-1,1), dx)$. 
\subsection{Large \texorpdfstring{$n$}{} limit without a parametrix solution}
In \cite{DKMVZ2} the authors were considering orthogonal polynomials with an exponential weight of the form $\E^{-Q(x)}$ on the real line, where $Q(x)$ was a polynomial of even degree and positive leading coefficient. To extract asymptotic results, they introduced a series of transformations of R-H problems and their solutions, starting with the solution 
\beq
Y := 2^{-n\sigma_3}X, \quad \sigma_3 := 
\begin{pmatrix}
1 & 0
\\
0 & -1
\end{pmatrix},
\eeq
for monic polynomials:
\beq \label{YUTSR}
Y \longrightarrow U \longrightarrow T \longrightarrow S \longrightarrow R
\eeq
In each step a conjugation or deformation step has been performed to obtain a new R-H problem. An analogous  procedure has been done by Kuijlaars et al.~in \cite{KMVV} for the modified Jacobi weight function $w^{\alpha,\beta}$ on $[-1,1]$ (see \eqref{Jacobi}). In this case the R-H problem for $U$ was not needed. We will not repeat the steps \eqref{YUTSR} here, but rather simply define the R-H problem for $S$ and $N$, for details see \cite{KMVV}. 

In our setting, we will assume that $\rho$ has an analytic continuation to a lense-shaped neighbourhood $\mathcal L$ of $(-1,1)$ as shown in Figure \ref{L}:
\vspace{-20pt}
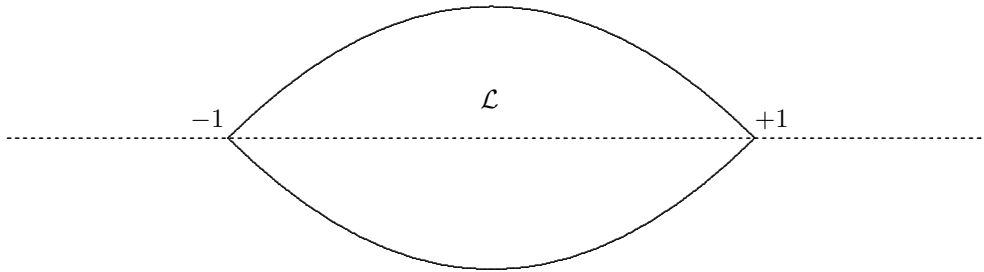
\begin{figure}[H]
\begin{picture}(7,5.2)

\curve(0,2.5, 3.5,0.75, 7,2.5)
\curve(0,2.5, 3.5,4.25, 7,2.5)

\curvedashes{0.05,0.05}
\curve(-3,2.5, 10,2.5)

\put(7, 2.65){$+1$}
\put(-0.5, 2.65){$-1$}
\put(3.35,2.9){$\mathcal L$}
\end{picture}
\vspace{-20pt}
\caption{\label{L} Neighbourhood $\mathcal L$.}
\end{figure}
Moreover, $|\rho^{\pm 1}|$ should remain bounded in $\mathcal L$, except possibly near the points $z = \pm 1$, where it can have the behaviour
\beq \label{condw}
|\rho(z)| = O(|z \pm 1|^{-1/\nu_++\varepsilon}), \quad |\rho(z)^{-1}| = O(|z \pm 1|^{-1/\nu_-+\varepsilon}),
\eeq
for two constants $\nu_\pm > 1$ and some $\varepsilon > 0$. This implies that for any smooth curve $\gamma$ going through $\mathcal L$ from $-1$ to $1$, we have
\beq
\Vert \rho \Vert_{L^{\nu_+}(\gamma,  dz)} < \infty, \quad \mbox{and} \quad \Vert \rho^{-1} \Vert_{L^{\nu_-}(\gamma,  dz)} <  \infty.
\eeq
The R-H problem for $S$ is defined on the contour $\Sigma := \Sigma_1 \cup \Sigma_2 \cup \Sigma_3$ as shown in Figure \ref{jumpcontour}.
\begin{figure}[H] 
\begin{picture}(7,5.2)
\put(0,2.5){\line(1,0){7.0}}
\put(3.5,2.5){\vector(1,0){0.2}}

\put(3.5,3.5){\vector(1,0){0.2}}
\put(3.5,1.5){\vector(1,0){0.2}}

\curve(0,2.5, 3.5,1.5, 7,2.5)
\curve(0,2.5, 3.5,3.5, 7,2.5)

\put(3.5, 3.7){$\Sigma_1$}
\put(3.5, 2.75){$\Sigma_2$}
\put(3.5, 1.7){$\Sigma_3$}

\put(5.5, 4.3){$\Omega_1$}

\put(5.1, 2.8){$\Omega_2$}

\put(5.1, 2){$\Omega_3$}
\curvedashes{0.05,0.05}
\curve(-3,2.5, 0,2.5)
\curve(7,2.5, 10,2.5)
\curve(0,2.5, 3.5,0.75, 7,2.5)
\curve(0,2.5, 3.5,4.25, 7,2.5)

\put(7, 2.65){$+1$}
\put(-0.5, 2.65){$-1$}
\put(2.4, 3,6){$\mathcal L$}
\end{picture}
\vspace{-20pt}
\caption{\label{jumpcontour} Jump contour for $S$.}
\end{figure}
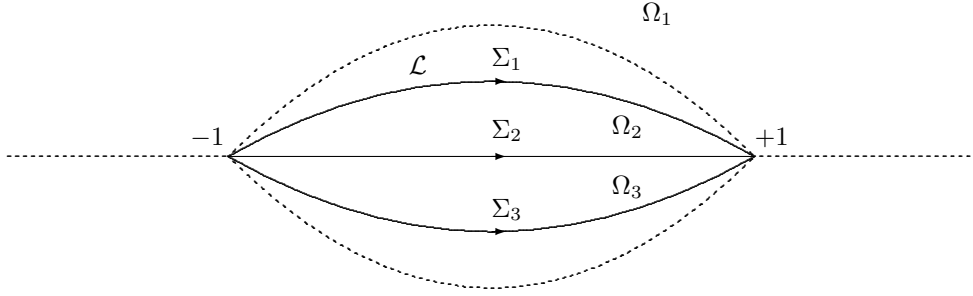
Here the region enclosed by $\Sigma_1$ and $\Sigma_2$ is denoted by $\Omega_2$ and assumed to be a subset of $\cL$, analogously for $\Omega_3$. Moreover, we set $\Omega_1 := \C \setminus \Omega_2 \cup \Omega_3$. 
The corresponding R-H problem for $S$ is now stated as follows: 
\\
\\
For any $n \in \mathbb{N}$ find a $2 \times 2$ matrix-valued function $S$ such that:
\vspace{10pt}
\begin{enumerate}[(i)]
\item $S(z,n)$ is analytic for $z \in \C \setminus \Sigma$,
\\
\item $S_+(k,n) = S_-(k,n) v_\cS(k,n),$ \ for $k \in \Sigma$
with:
\beq \label{jumpS}
v_\cS(k,n) =
\begin{cases}
\begin{pmatrix}
1 & 0
\\
\rho(k)^{-1} \varphi(k)^{-2n} & 1
\end{pmatrix}, \quad \mbox{for }k \in \Sigma_1 \cup \Sigma_3,
\\
\begin{pmatrix}
0 & \rho(k)
\\
-\rho(k)^{-1} & 0
\end{pmatrix}, \hspace{24pt} \mbox{for } k \in \Sigma_2 = (-1,1),
\end{cases} 
\eeq
\\
\item $S(z,n) = \id + O(z^{-1}), \hspace{7pt} \mbox{as } z \rightarrow \infty$,
\end{enumerate}
\vspace{10pt}
with at most square integrable singularities at $z = \pm 1$. Here
\beq
\varphi(z) := z + \sqrt{z^2-1}
\eeq
maps $\C \setminus [-1,1]$ biholomorphically to the exterior of the unit disc. In particular $|\varphi(z)| > 1$ for $z \in \C \setminus [-1,1]$ and $|\varphi(z)| \rightarrow 1$ as $z$ approaches $[-1,1]$. Looking at the jump matrix \eqref{jumpS} and observing that
\beq
\log \varphi(z) = O(\sqrt{|z\mp1|}), \quad \mbox{as} \ z \rightarrow \pm 1,
\eeq
we identify $z = \pm 1$ as the oscillatory points with $\chi = 1/2$ (cf.~\eqref{vkappa}). 

The solution $S$ is related to the solution $X$ from Theorem \ref{orthogonalMatrix} via
\beq
S(z,n) = 
\begin{cases}
X(z,n) \varphi(z)^{-n\sigma_3}, \quad z \in \Omega_1,
\\
X(z,n) \varphi(z)^{-n\sigma_3}
\begin{pmatrix}
1 & 0
\\
-\rho(z)^{-1} \varphi(z)^{-2n} & 1
\end{pmatrix},
\quad z \in \Omega_2,
\\
X(z,n) \varphi(z)^{-n\sigma_3}
\begin{pmatrix}
1 & 0
\\
\rho(z)^{-1} \varphi(z)^{-2n} & 1
\end{pmatrix}, \hspace{8pt} \quad z \in \Omega_3.
\end{cases}
\eeq
and satisfies
\beq \label{limitS}
\lim_{z\rightarrow\infty} S(z,n) = \lim_{z\rightarrow \infty} X(z,n)\varphi^{-n\sigma_3} = \id.
\eeq
The corresponding model problem has the same normalization at infinity, but only a jump condition on $(-1,1)$, 
\beq
N_+(x) = N_-(x)
\underbrace{
\begin{pmatrix}
0 & \rho(x)
\\
-\rho(x)^{-1} & 0
\end{pmatrix}}_{=:v_\cN(x)}, \quad x \in (-1,1),
\eeq
and is independent of $n$. The explicit solution given in \cite[Eq.~5.5]{KMVV} has the form
\beq
N(z) = D_\infty^{\sigma_3}
\begin{pmatrix}
\dfrac{a(z)+a(z)^{-1}}{2} && \dfrac{a(z)- a(z)^{-1}}{2\I}
\\
\dfrac{a(z)- a(z)^{-1}}{-2\I} && \dfrac{a(z)+a(z)^{-1}}{2}
\end{pmatrix} D(z)^{-\sigma_3}
\eeq
where $D(z)$ is the Szeg\H{o} function associated to $\rho$ \cite{GS}:
\beq
D(z) := \exp \Bigg( \dfrac{\sqrt{z^2-1}}{2 \pi}\int_{-1}^1 \dfrac{\log \rho(x)}{\sqrt{1-x^2}} \dfrac{dx}{z-x} \Bigg), \hspace{7pt} z \in \C \setminus [-1,1], 
\eeq
\beq
D_\infty := \lim\limits_{z \rightarrow \infty} D(z) = \exp \Bigg( \dfrac{1}{2 \pi}\int_{-1}^1 \dfrac{\log \rho(x)}{\sqrt{1-x^2}} \hspace{3pt} dx \Bigg)
\eeq
and
\beq
a(z) := \Big( \dfrac{z-1}{z+1} \Big)^{1/4} 
\eeq
with a branch cut on $[-1,1]$ and $a(\infty) = 1$. We note that
\beq
\Sigma^{\exp} := \Sigma_1 \cup \Sigma_3, \quad \quad \Sigma^{mod}:= \Sigma_2 = (-1,1).
\eeq
\begin{remark}
It is important to note that $v_\cN$ might not be uniformly bounded and $N$ might not satisfy condition \eqref{Nkappa}. This is because of the singular behaviour of the weight function $\rho$, and hence $D$, near the endpoints \eqref{condw}. However, we will still use H\"older's inequality to get similar results as in Theorem \ref{mainTheorem}. 
\end{remark}
Next we want to arrive at estimates for $\Vert S_- w_\cS N \Vert_{L^1(\Sigma^{exp}, dz)}$. As $N$ has no jumps on $\Sigma^{exp}$, the $\pm$ subscripts can be left out (in the following inequality we abbreviate $\Vert \ . \ \Vert_{L^p(\Sigma^{exp}, dz)}$ by $\Vert \ . \ \Vert_p$ and so on):
\beq \label{holder}
\begin{gathered}
\Vert S_- w_\cS N^{-1} \Vert_{L^1(\Sigma^{exp}, dz)} = \Bigg\Vert X \varphi^{-n\sigma_3} \begin{pmatrix}
0 & 0 
\\
\rho^{-1}\varphi^{-2n} & 0
\end{pmatrix}  N^{-1} \Bigg\Vert_{L^1(\Sigma^{exp}, dz)}
\\
\lesssim \ \Vert X \varphi^{-n\sigma_3} \Vert_{L^p(\Sigma^{exp}, dz)} \ \Vert \rho^{-1} \Vert_{L^\vartheta(\Sigma^{exp}, dz)} \ 
\\
\times \
\Vert  \varphi^{-2n} (z\pm 1)^{-1/4}
\Vert_{L^\tau(\Sigma^{exp}, dz)}\ 
\Vert D^{\sigma_3} \Vert_{L^\omega(\Sigma^{exp}, dz)},
\end{gathered}
\eeq
with 
\beq \label{pqrs}
\dfrac{1}{p} + \dfrac{1}{\vartheta} + \dfrac{1}{\tau} + \dfrac{1}{\omega} = 1.
\eeq
We have used for simplicity $\varphi^{-2n} (z\pm 1)^{-1/4}$, which has the same growth behaviour as the more complicated expression
\beq
\Bigg\Vert \begin{pmatrix}
0 & 0
\\
\varphi^{-2n} & 0
\end{pmatrix} 
\begin{pmatrix}
\dfrac{a(z)+a(z)^{-1}}{2} && \dfrac{a(z)- a(z)^{-1}}{-2\I}
\\
\dfrac{a(z)- a(z)^{-1}}{2\I} && \dfrac{a(z)+a(z)^{-1}}{2}
\end{pmatrix}D_\infty^{-\sigma_3}
\Bigg\Vert_{L^\tau(\Sigma^{exp},dz)}.
\eeq
Using \eqref{vkappa} with $\beta = 1/2$ and $q = \tau$, we see that
\beq
\Vert \varphi^{-2n} (z\pm 1)^{-1/4} \Vert_{L^\tau(\Sigma^{exp},dz)} = O(n^{-2/\tau + 1/2}).
\eeq
Hence, we should try to choose $\tau$ as small as possible, which translates into maximizing $p$, $\vartheta$ and $\omega$ under the constraint that all the corresponding terms in \eqref{holder} remain bounded. 

Next we want to show that $\Vert X \varphi^{-n\sigma_3} \Vert_{L^p(\Sigma^{exp},dz)}$ remains bounded for an appropriate $p \geq 1$. Using Lemma \ref{extension} together with \eqref{limitS}, it is enough to establish the $L^p((-1,1), dx)$-boundedness of $X \varphi^{-n\sigma_3}_\pm$, which, as $|\varphi_\pm(x)| = 1$ for $x \in (-1,1)$, is equivalent to showing that
\beq
\Vert X_\pm \Vert_{L^p((-1,1), dx)} \leq C_1 < \infty.
\eeq
 Looking at the components of $X$ in \eqref{X} and using the $L^p$-boundedness of the Cauchy boundary operators $\mathcal C_\pm^{(-1,1)}$ on $L^p((-1,1),  dx)$ for $p \in (1,\infty)$, we need to show that
\beq
\Vert p_n \Vert_{L^p((-1,1), dx)} \leq C_2, \quad  \Vert p_n \rho \Vert_{L^p((-1,1), dx)} \leq C_3, \quad C_2, C_3 < \infty.
\eeq
For both cases we can use H\"older's inequality and theorem \ref{szego} which tells us that
\beq
\Vert p_n \sqrt{\rho} \Vert_{L^a((-1,1), dx)} \leq C_4 < \infty.
\eeq
for $a = 2$ and thus automatically for $a \leq 2$, as we integrate over a finite interval.
Using this we can write
\begin{align}
\Vert p_n \Vert_{L^p((-1,1), dx)} &\leq \Vert p_n \sqrt{\rho} \Vert_{L^a((-1,1), dx)} \Vert  \sqrt{\rho}^{-1} \Vert_{L^b((-1,1), dx)}
\\
\Vert p_n \rho \Vert_{L^p((-1,1), dx)} &\leq \Vert p_n \sqrt{\rho} \Vert_{L^{a'}((-1,1), dx)} \Vert  \sqrt{\rho} \Vert_{L^{b'}((-1,1), dx)}
\end{align}
with
\beq 
\dfrac{1}{p} = \dfrac{1}{a} + \dfrac{1}{b} = \dfrac{1}{a'} + \dfrac{1}{b'}.
\eeq
 As we want to maximize $p$, we choose $a = a' = 2$. The condition \eqref{condw} tells us that we can take $b \leq 2\nu_-$ and $b' \leq 2\nu_+$. Again, maximizing $p$, we choose
\beq
b = b' = 2\nu_0 := 2\min\lbrace \nu_+, \nu_- \rbrace 
\eeq
which gives us
\beq
p := \dfrac{2\nu_0}{1+\nu_0}.
\eeq
The condition $p > 1$ translates to $\nu_0 > 1$, which was assumed right after \eqref{condw}.

Next, let us consider the term $\Vert \rho^{-1} \Vert_{L^\vartheta(\Sigma^{exp},dz)}$. This is the simplest case, as \eqref{condw} implies that we can choose $\vartheta := \nu_-$.

The term $\Vert D^{\sigma_3} \Vert_{L^\omega(\Sigma^{exp},dz)}$ is more challenging. Recall that the Szeg\H{o} function satisfies \cite[Eq.~2.14--15]{DC}
\beq
D_+(x) D_-(x) = \rho(x), \quad x \in (-1,1)
\eeq
and
\beq
\overline{D(z)} = D(\overline{z}), \quad \C \setminus [-1,1]. 
\eeq
These two identities imply
\beq
|D_\pm(x)| = \sqrt{\rho(x)} \quad\mbox{and} \quad D_+(x) - D_-(x) = 2 \imaginary(D_+(x)), \quad x \in (-1,1).
\eeq
Hence, $D$ satisfies an additive R-H problem, with 
\beq
\Vert D_+ - D_- \Vert_{L^\omega((-1,1), dx)} \leq \Vert \sqrt{\rho} \Vert_{L^\omega((-1,1), dx)}. 
\eeq
As $D$ has a limit at infinity, namely $D_\infty$, we can apply lemma \ref{extension} to conclude that
\beq
\Vert D \Vert_{L^\omega(\Sigma^{exp},dz)} = O(\Vert \sqrt{\rho} \Vert_{L^\omega((-1,1),dx)}).
\eeq
From \eqref{condw} it follows that we must have $\omega \leq 2\nu_+$. The same argument works with $D^{-1}$, and in the end we can choose $\omega := 2\nu_0$.

We have computed the optimal values for $p$, $\vartheta$ and $\omega$:
\begin{align}
p = \dfrac{2\nu_0}{1+\nu_0}, 
\hspace{40pt}
\vartheta = \nu_-, 
\hspace{40pt}
\omega = 2\nu_0.
\end{align}
A quick calculation using the relation \eqref{pqrs} shows that
\beq
\tau = \dfrac{2\nu_0 \nu_-}{\nu_0 \nu_--2(\nu_0+\nu_-)},
\eeq
which implies
\beq
\Vert  (z\pm 1)^{-1/4} \varphi^{-2n}\Vert_{L^\tau(\Sigma^{exp},dz)} = O(n^{-\lambda}),
\eeq
with
\beq \label{lambda}
\lambda := \dfrac{1}{2}- \dfrac{2(\nu_0+\nu_-)}{\nu_0\nu_-}.
\eeq
As all the other terms on the right-hand side of \eqref{holder} remain bounded as $n \rightarrow \infty$, we conclude: 
\beq
\Vert S_-w_\cS N^{-1} \Vert_{L^1(\Sigma^{exp},dz)} = O(n^{-\lambda}). 
\eeq
To ensure that $\lambda$ is positive it is enough to assume that 
\beq \label{nu1}
\nu_0 > 8,
\eeq
and in the case $\nu_0 = \nu_+ \in (4,8)$, one would need
\beq \label{nu2}
\nu_- > \dfrac{4\nu_+}{\nu_+-4}.
\eeq
For positive $\lambda$ we can conclude that 
\beq \label{SNasym}
S(z, n) = (\id + O(n^{-\lambda} \dist(z, \Sigma^{exp})^{-1}))N(z).
\eeq
In particular
\begin{align} \label{asymresult}
X(z,n) &= (\id + O(n^{-\lambda} \dist(z, \Sigma^{exp})^{-1}))N(z)\varphi(z)^{n\sigma_3}, &z \in \Omega_1,
\\\label{asymresult2}
X(z,n) &= (\id + O(n^{-\lambda} \dist(z, \Sigma^{exp})^{-1}))N(z)
\\\nonumber
& \hspace{69pt} \times \begin{pmatrix}
1 & 0
\\
\rho^{-1}(z)\varphi^{-2n}(z) & 1
\end{pmatrix}
\varphi(z)^{n\sigma_3}, &z \in \Omega_2,
\\\label{asymresult3}
X(z,n) &= (\id + O(n^{-\lambda} \dist(z, \Sigma^{exp})^{-1})N(z)
\\\nonumber
& \hspace{69pt} \times
\begin{pmatrix}
1 & 0
\\
-\rho^{-1}(z)\varphi^{-2n}(z) & 1
\end{pmatrix}
\varphi(z)^{n\sigma_3},  &z \in \Omega_3.
\end{align}
We can state our main result concerning orthogonal polynomials:
\begin{theorem} \label{theorem}
Let the weight function $\rho$ have an analytic continuation to a lense-shaped neighbourhood $\cL$ of $(-1,1)$  and satisfy condition \eqref{condw}, with $\nu_\pm$ fulfilling either \eqref{nu1} or \eqref{nu2}. Moreover, let $U$ be a compact subset of the Riemann sphere contained in $\mathbb{C} \setminus [-1,1] \cup\lbrace \infty \rbrace$ and $V$ be a compact set contained in $\cL$. Then, we have for the $n$-th orthogonal polynomial with leading coefficient $2^n$:
\begin{align} \label{asymU}
p_n(z) = \dfrac{D_\infty \varphi(z)^{n+1/2}}{\sqrt{2}D(z)(z^2-1)^{1/4}} +O(n^{-\lambda} z^{-1}\varphi(z)^n), \quad z \in U,
\end{align}
and
\begin{align} \label{asymV}
p_n(z) = \dfrac{1}{\sqrt{2}(z^2-1)^{1/4}}\Big(\dfrac{D_\infty}{D(z)}\varphi(z)^{n+1/2} \pm \I \dfrac{D_\infty D(z)}{\rho}\varphi(z)^{-n-1/2}\Big)
\notag\\
+O(n^{-\lambda}\varphi(z)^n), \quad z \in V,
\end{align}
with the $+(-)$ sign in the case of $z$ in the upper(lower)-half plan, and both signs giving the same result for $z \in(-1,1)$. The constant $\lambda$ is given by \eqref{lambda}.
\end{theorem}
\begin{proof}
First observe the two identities
\beq \label{a1}
\dfrac{a(z)+a(z)^{-1}}{2} = \dfrac{\varphi(z)^{1/2}}{\sqrt{2}(z^2-1)^{1/4}}
\eeq
and
\beq \label{a2}
\dfrac{a(z)-a(z)^{-1}}{2\I} = \I \dfrac{\varphi(z)^{-1/2}}{\sqrt{2}(z^2-1)^{1/4}}.
\eeq

Next, let us consider the asymptotics in $U$. As $U$ is assumed to be compact on the Riemann sphere, it must be a finite distance away from $[-1,1]$. In particular, we can choose the contour $\Sigma^{\exp} \subset \cL$ such that
\beq
\dist(U, \Sigma^{exp}) > 0.
\eeq
It follows that for $z \in U$, we have that asymptotic terms of order $O(z^{-1})$ and $O(\dist(z,\Sigma^{exp})^{-1})$ become equivalent. Again by compactness of $U$, we see that the functions $D^{\pm 1}$ and $a^{\pm1}$ are bounded in $U$. With this information, \eqref{asymU} is obtained by multiplying out \eqref{asymresult} and using $p_n(z) = X_{11}(z,n)$ from Theorem \ref{orthogonalMatrix}.

For the set $V$ we can again choose $\Sigma^{exp} \subset \cL$ such that
\beq
\dist(V,  \Sigma^{exp}) > 0.
\eeq
As $V$ is bounded, asymptotic terms of order $O(\dist(z, \Sigma^{exp})^{-1})$ and $O(1)$ become equivalent. Similar to before the functions $D^{\pm1}$ and $a^{\pm1}$ are uniformly bounded, as they are continuous in $V \setminus (-1,1)$ and take continuous limits on $V \cap (-1,1)$. Multiplying out \eqref{asymresult2} and \eqref{asymresult3} gives then \eqref{asymV}. Moreover, using for $x \in (-1,1)$,
\begin{align}
    (x^2-1)^{1/4}_+ &= \I (x^2-1)^{1/4}_-
    \notag\\
    D_+(x)D_-(x) &= \rho(x)
    \\\nonumber
    \varphi_+(x)\varphi_-(x) &= 1
\end{align}
one can verify that both signs in \eqref{asymV} give the same result on $V \cap (-1,1)$.
\end{proof}
\begin{remark}
The leading terms in \eqref{asymU} and \eqref{asymV} have been obtained by Bernstein and Szeg\H{o} in \emph{\cite[Thm.~12.1.1--4]{GS}}. However, the R-H method allows for more explicit bounds on the error terms. 
\end{remark}
\subsection{R-H problem with constant jump matrices} \label{constantMatrices}
Finally, let us explain why the choice of weight function makes a reformulation of the local parametrix problem to a R-H problem with constant jump matrices, as found in \cite[Sect.~6]{KMVV}, impossible. In that paper the authors considered the following local parametrix problem for $P$ in a small but fixed disc $U_\delta$ of radius $\delta$ around $z = 1$ (and later analogously around $z = -1$):
\\
\\
Find a $2 \times 2$ matrix-valued function $P$, such that
\vspace{10pt}
\begin{enumerate}[(i)]
    \item $P(z,n)$ is analytic for $z \in U_\delta \setminus \Sigma$,
    \\
    \item $P_+(k,n) = P_-(k,n) v_\cS(k,n)$, \ for  $ k \in \Sigma \cap U_\delta$,
    \\
    \item $P(k,n)N^{-1}(k) = \id + o(1)$, \ uniformly for $k \in  \partial U_\delta$ as $n\rightarrow \infty$.\footnote{In \cite{KMVV} an error term of order $O(n^{-1})$ is used instead of $o(1)$.} 
\end{enumerate}
\vspace{10pt}
Hence, $P$ should satisfy locally around $z = 1$ the same R-H problem as $S$, but the normalization at infinity is changed to a matching condition with $N$. To transform this R-H problem to an explicitly solvable one, a further conjugation step is needed to make the jump matrices independent of $k \in \Sigma$. This is a crucial step that cannot be performed in our case. Assume for a moment that $\rho$ has an analytic nowhere vanishing continuation in $U_\delta$. In \cite[Eq.~6.7]{KMVV} the authors define the function $W$ by\footnote{In \cite{KMVV} a Jacobi-weight prefactor $(1-x)^\alpha(1+x)^\beta$ was considered, which is not treated in our paper, hence we study the case $\alpha = \beta = 0$.}
\beq \label{W}
W(z) := \sqrt{\rho(z)}, \quad z \in U_\delta.
\eeq 
This function can be used to construct $P^{(1)}$:
\beq
P^{(1)}(z,n) := P(z,n)\varphi^{n\sigma_3}(z) W(z)^{\sigma_3}, \quad U_\delta.
\eeq
 Then $P^{(1)}$ satisfies the following R-H problem:
\\
\\
Find a $2 \times 2$ matrix-valued function $P^{(1)}$, such that
\vspace{10pt}
\begin{enumerate}[(i)]
    \item $P^{(1)}(z,n)$ is analytic in $z\in U_\delta \setminus \Sigma$,
    \\
    \item $P_+^{(1)}(k,n) = P_-^{(1)}(k,n) v_\cP(k,n)$, \ for  $ k \in \Sigma \cap U_\delta$,
    \\
    \item $P^{(1)}(k,n)\big(N(k)\varphi(k)^{n\sigma_3}W(k)^{\sigma_3}\big)^{-1} = \id + o(1)$, 
    \vspace{5pt}
    \\
     uniformly for $k \in  \partial U_\delta$ as $n\rightarrow \infty$.
\end{enumerate}
\vspace{10pt}
with
\beq
v_\cP(k) =
\begin{cases}
\begin{pmatrix}
1 & 0
\\
1 & 1
\end{pmatrix}, \hspace{11pt} \mbox{for} \ k \in \Sigma^{exp} \cap U_\delta,
\\
\begin{pmatrix}
0 & 1
\\
-1 & 0
\end{pmatrix}, \ \mbox{for} \ k \in \Sigma^{mod} \cap U_\delta.
\end{cases}
\eeq
After an $n$-dependent change of variables $z \rightarrow \zeta$, the matching condition (iii) for $P^{(1)}$ can be transformed into a normalization at infinity for the matrix-valued function 
\beq
\Psi(\zeta) := P^{(1)}(z(\zeta)). 
\eeq
The corresponding jump contour $\Sigma^\cB := \Sigma_1^\cB \cup \Sigma^\cB_2 \cup \Sigma^\cB_3$ consists of three rays emanating from the origin as in Figure \ref{besselcontour}:
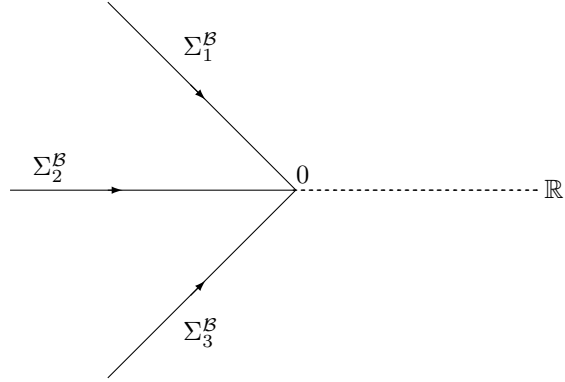
\begin{figure}[H] 
\begin{picture}(7,5.2)
\put(-0.3,2.5){\line(1,0){3.8}}
\put(1,2.5){\vector(1,0){0.2}}

\put(1,0){\line(1,1){2.5}}
\put(2.1,1.1){\vector(1,1){0.2}}

\put(1,5){\line(1,-1){2.5}}
\put(2.1,3.9){\vector(1,-1){0.2}}

\put(2,4.3){$\Sigma_1^\cB$}
\put(0,2.7){$\Sigma_2^\cB$}
\put(2,0.5){$\Sigma_3^\cB$}

\put(6.8,2.4){$\R$}
\put(3.5,2.6){$0$}

\curvedashes{0.05,0.05}
\curve(3.4,2.5, 6.7,2.5)
\end{picture}
\caption{\label{besselcontour} Contour for the Bessel R-H problem}
\end{figure}
and the jump matrix $v_\cB$ has the form
\beq
v_\cB(\zeta) =
\begin{cases}
\begin{pmatrix}
1 & 0
\\
1 & 1
\end{pmatrix}, \hspace{11pt} \mbox{for} \ \zeta \in \Sigma_1^\cB \cup \Sigma_3^\cB,
\\
\begin{pmatrix}
0 & 1
\\
-1 & 0
\end{pmatrix}, \ \mbox{for} \ \zeta \in \Sigma_2^\cB.
\end{cases}
\eeq
The R-H problem for $\Psi$ is stated as follows:
\\
\\
Find a $2 \times 2$ matrix-valued function $\Psi$, such that
\vspace{10pt}
\begin{enumerate}[(i)]
    \item $\Psi(\zeta)$ is analytic for $\zeta \in \C \setminus \Sigma^\cB$,
    \\
    \item $\Psi_+(\zeta) = \Psi_-(\zeta) v_\cB(\zeta)$, \ for  $ \zeta \in \Sigma^\cB$,
    \\
    \item $\Psi(\zeta) \rightarrow (2\pi \zeta^{1/2})^{-\sigma_3/2} \dfrac{1}{\sqrt{2}} 
    \bigg(\begin{pmatrix}
    1 & i
    \\
    i & 1
    \end{pmatrix}+o(1)\bigg)\E^{2\zeta^{1/2}\sigma_3}$
    ,  \quad  uniformly as $\zeta \rightarrow \infty$.
\end{enumerate}
\vspace{10pt}
This R-H problem can be regarded as a limit as $n\rightarrow \infty$ of the local parametrix problem. We will refer to such R-H problems in this context as \emph{R-H problems with constant jump matrices}. The preceding example is solved explicitly using the Bessel functions \cite[Eq.~6.23--25]{KMVV}. In the Appendix we provide the heuristics by which the solutions to R-H problems with constant jump matrices can be found, using the example of the Airy R-H problem. 

As $\rho$ is not assumed to have an analytic continuation to a disc $U_\delta$ around $z = 1$ ($z = -1$), the function $W$ in \eqref{W} cannot be defined in general. Hence, the reformulation of the parametrix problem as a R-H problem with constant jump matrices, at least with the usual approach, is not possible. However, we point out that under the assumptions of Theorem \ref{theorem} the local parametrix problem for $P$ indeed has a solution.
\begin{theorem}
The exact solution $S$, satisfies the local parametrix R-H problem for $P$.
\end{theorem}
\begin{proof}
The matrix-valued function $S$ satisfies trivially condition (i) and (ii) for $P$. The only remaining condition (iii) is the matching condition:
\beq
S(k,n)N(k,n)^{-1} = R(k,n) = \id + o(1), \quad \mbox{for } k \in \partial U_\delta.
\eeq
uniformly as $n \rightarrow \infty$. We can substitute $O(n^{-\lambda}\dist(k,\Sigma^{exp})$ for the error term $o(1)$, because of \eqref{SNasym}. However, by deforming the contour $\Sigma^{exp}$ we see that the points where $\Sigma^{exp}$ and $U_\delta$ meet are movable, and that the error term is in fact uniform on $\partial U_\delta$, meaning
\beq
S(k,n)N(k,n)^{-1} = R(k,n) = \id + O(n^{-\lambda}), \quad \mbox{for } k \in \partial U_\delta,
\eeq
uniformly as $n \rightarrow \infty$. Hence, $S$ is a solution to the local parametrix R-H problem.
\end{proof}
The solution $S$ is unique in the following sense: For any other solution $\widetilde P$, the matrix-valued function $H:= \widetilde P S^{-1}$ will be analytic in $U_\delta$, and satisfy
\beq
H(k,n) = \widetilde P(k,n)S(k,n)^{-1} = \id+ o(1) \quad \mbox{for } k \in \Sigma \cap U_\delta,
\eeq
uniformly as $n \rightarrow \infty$. By the maximum principle for analytic function, we therefore know that
\beq
H(z,n) = \widetilde P(z,n) S(z,n)^{-1} = \id+  o(1) \quad \mbox{for } z \in  U_\delta,
\eeq
uniformly as $n \rightarrow \infty$. Hence we see that $\widetilde P$ has the form
\beq
\widetilde P(z,n) = H(z,n)S(z,n), \quad z \in U_\delta,
\eeq
where $H(z,n)$ is a sequence of matrix-valued function with analytic entries, such that it converges uniformly to the identity matrix as $n \rightarrow \infty$. Moreover, every solution of the local parametrix problem can be obtained in this way. 
\section{Discussion}
We have shown that the explicit construction of a local parametrix solution is not necessary for a rigorous R-H analysis, provided an a priori $L^p$-estimate of the exact solution $S$ is known. Our method has been illustrated on the example of orthogonal polynomials on $[-1,1]$ with a new class of weight functions $\rho$. We impose analytic continuation of $\rho$ to a lense-shaped neighbourhood of $(-1,1)$ and a growth condition at the endpoints (cf.~Figure \ref{L} and Eq.~\eqref{condw}). In particular, we do not require any sort of analytic continuation around $x = \pm 1$. This feature differentiates our class of admissible weight functions, from those that lead to solvable parametrix problems.
In fact, the assumptions on $\rho$ made the reformulation of the local parametrix problem to an explicitly  solvable R-H problem with piecewise constant jump matrices, as done in \cite[Sect.~6]{KMVV}, impossible. 

However, the error bound obtained that way is in general worse than the actual error term. In \cite{KMVV} the authors show
\beq \label{expansion}
S(z,n) =\Big(\id + \sum_{k = 1}^\ell \dfrac{R_k(z)}{n^k} + O(n^{-\ell-1})\Big) N(z), \quad \mbox{as } n \rightarrow \infty, \ \ell \in \N_+, 
\eeq
uniformly away from $x = \pm 1$. They considered modified Jacobi weight functions of the form \eqref{Jacobi}. The series expansion stems from the series expansion of the Bessel functions, which are contained in the solution of the parametrix problem.
Our approach would only result in $O(n^{-1/2})$ without a series expansion, for weight functions $\rho$ satisfying that $|\rho^{\pm 1}|$ is uniformly bounded. This case corresponds to $\nu_\pm \rightarrow \infty$ in \eqref{condw}. This is unsurprising as the a priori $L^p$-estimate contains very limited information on the local structure of the exact solution $S$ around the oscillatory points. 

Whether an expansion of the form \eqref{expansion} holds in our case, seems to be an open problem. Also the optimal error bound in \eqref{asymU} and \eqref{asymV} is, to the best of our knowledge, unknown.
These questions relate to one of the main motivations for obtaining large $n$ asymptotics of orthogonal polynomials, namely the study of eigenvalue statistics for ensembles of random matrices. This field of study has been initiated by Wigner \cite{Wigner55}. The statistics in the bulk of the spectrum have been further studied for special cases by Dyson in \cite{Dyson1962}, \cite{Dyson1970} and Mehta in \cite{Mehta71} and confirmed instances of the Wigner--Dyson--Mehta universality conjecture. This universality conjecture states that the local statistics in the bulk of the spectrum depend only on the type of the ensemble, which is either unitary, orthogonal or symplectic. In the unitary case, this conjecture has been proven for the classical Hermite, Laguerre and Jacobi ensembles (\cite{Dyson1970}, \cite{NW1}, \cite{SNW}). Similar results hold for Wigner matrices, but require different methods (\cite{EPRSY}, \cite{ERSTY}, \cite{EY17}, \cite{TV}). 

Our results show universality in the bulk, but also have interesting connections with universality near the edge. To understand why, we have to look at known results. Here one has to distinguish between the soft edge (\cite{BI}, \cite{PD}, \cite{DKMVZ2}, \cite[Sect.~3]{Forrester}) and the hard edge (\cite[Sect.~2]{Forrester}, \cite{KMVV}, \cite{KV}, \cite{NW93}). While the former leads to local statistics described by the Airy kernel, the latter leads to the Bessel kernel. Both kernels have been further studied by Tracy and Widom (\cite{TWAiry}, \cite{TWBessel}). Pollaczek--type weight functions are in turn related to solutions to the Painlev\'e III equation \cite{CI}. The corresponding Deift--Zhou nonlinear steepest descent analysis has been performed in \cite{BMM}, \cite{MM}, \cite{XDZ}, (see \cite{CCF} for the Jacobi case). The Airy and Bessel kernels are both special limiting cases of the Painlev\'e III kernel. Interestingly, showing this leads to auxiliary local parametrix problems within the standard local parametrix problems \cite[Sect.~5, 6]{XDZ}. The case of edge universality for Wigner matrices has been studied in \cite{BEYBeta}, \cite{SO99}, \cite{TVEdge}.

Using the R-H method, the cited results on edge universality (except the ones concerning Wigner matrices) can be obtained through  the explicit solution of a R-H problem with constant jump matrices. This R-H problem is a limit of local parametrix problems as the polynomial degree $n$ goes to infinity. However, as shown in Section \ref{constantMatrices}, for weight functions that do not have an analytic continuation outside the lense $\cL$ (cf.~Figure \ref{L}) one cannot formulate a corresponding limiting R-H problem with constant jump matrices as it is usually done. The natural question arises, whether the local parametrix problems converge in any other sense as the degree $n$ goes to infinity. A hypothetical limiting R-H problem would likely determine the behaviour of the eigenvalue statistics near the edge of the spectrum. 

Our work suggests that such a limiting R-H problem might not exist, as the nonlinear steepest descent method can be rigorously applied without computing any limit of the local parametrix problems. This leaves open the possibility that there are ensembles of random matrices, corresponding to the weight functions described in this paper, that exhibit universality in the bulk but not at the edge of the spectrum. More research is needed to formalize and prove or disprove this statement.

Another future challenge would be proving the a priori $L^p$-estimate in different settings, for example long-time asymptotics of nonlinear PDEs solvable via scattering theory or Placherel--Rotach asymptotics of orthogonal polynomials on the real line. 
\appendix
\section{Airy R-H problem}
The Airy R-H problem occurs in the theory of orthogonal polynomials, inverse scattering and Painlev\'e transcendents (\cite{BI}, \cite[Ch.~5]{PD}, \cite{DKMVZ3}, \cite{DKMVZ2}, \cite{Painleve},  \cite{ParametrixFinal}). We choose this example to illustrate the heuristics by which an explicit solution can be found, as the underlying Airy differential equation is particularly simple. The contour $\Sigma^\cA$ is a union of four rays $\Sigma_i^\cA$, $i = 1,..,4$, in the complex plane that meet at the origin as is shown in Figure \ref{airycontour}.
\begin{figure}[H] 
\begin{picture}(7,5.2)
\put(-0.3,2.5){\line(1,0){7.4}}
\put(5.5,2.5){\vector(1,0){0.2}}
\put(1,2.5){\vector(1,0){0.2}}

\put(1,0){\line(1,1){2.5}}
\put(2.1,1.1){\vector(1,1){0.2}}

\put(1,5){\line(1,-1){2.5}}
\put(2.1,3.9){\vector(1,-1){0.2}}

\put(6,2.7){$\Sigma^\cA_1$}
\put(2,4.3){$\Sigma^\cA_2$}
\put(0,2.7){$\Sigma^\cA_3$}
\put(2,0.5){$\Sigma^\cA_4$}

\put(4.8,4.2){$\Omega_1$}
\put(0,4){$\Omega_2$}
\put(0,1){$\Omega_3$}
\put(4.8,0.5){$\Omega_4$}

\put(3.5,2.55){$0$}
\put(7.2,2.4){$\R$}

\end{picture}
\caption{\label{airycontour} Contour for the Airy parametrix problem}
\end{figure}
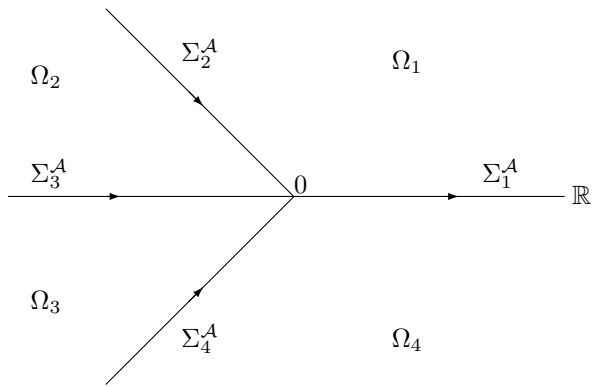
To the above contour we associate the following piecewise constant jump matrix $v_\cA$:
\begin{align}
v_\cA(k) := 
\begin{cases}
\begin{pmatrix}
1 & -\I
\\
0 & 1
\end{pmatrix}, \ &k \in \Sigma_1^\cA,
\\
\begin{pmatrix}
1 & 0
\\
\I & 1
\end{pmatrix}, \ &k \in \Sigma^\cA_2,
\\
\begin{pmatrix}
0 & -\I
\\
-\I & 0
\end{pmatrix}, \ &k \in \Sigma^\cA_3,
\\
\begin{pmatrix}
1 & 0
\\
\I & 1
\end{pmatrix}, \ &k \in \Sigma^\cA_4.
\end{cases}
\end{align}
The Airy R-H problem is stated as follows:
\\
\\
Find a $2 \times 2$ matrix-valued function $\Upsilon$ such that
\vspace{10pt}
\begin{enumerate}[(i)]
\item $\Upsilon(z)$ is analytic for $z \in \C \setminus \Sigma^\cA$, 
\\
\item $\Upsilon_+(k) = \Upsilon_-(k)v_\cA(k)$, \ for $k \in \Sigma^\cA$
\\
\item $\Upsilon(z) \rightarrow z^{-\sigma_3/4} \dfrac{\E^{\pi \I/12}}{2\sqrt{\pi}}
\bigg(\begin{pmatrix}
1 & 1
\\
-1 & 1
\end{pmatrix} +o(1)\bigg)
\E^{-(2/3) z^{3/2}\sigma_3}$, \quad uniformly as $z \rightarrow \infty.
$ 
\end{enumerate}
\vspace{10pt}
Here $z^{3/2}$ has a branch cut on $\R_-$. 

The jump matrices satisfy the cyclic condition, meaning their cyclic product when taking into account the contour orientation evaluates to the identity matrix:
\beq
\begin{pmatrix}
1 & -\I
\\
0 & 1
\end{pmatrix}
\begin{pmatrix}
1 & 0
\\
\I & 1
\end{pmatrix}^{-1}
\begin{pmatrix}
0 & -\I
\\
-\I & 0
\end{pmatrix}^{-1}
\begin{pmatrix}
1 & 0
\\
\I & 1
\end{pmatrix}^{-1}= \id.
\eeq
In \cite[Sect.~3.2]{QW} it is shown that this implies that solutions without any normalization at infinity are in a bijective correspondence with entire matrix-valued functions. More specifically, provided a solution $\Upsilon$ exists, there is a unique entire matrix-valued function $E$ with $\Upsilon(z) = E(z)$ for $z \in \Omega_1$, such that: 
\beq \label{PE}
\Upsilon(z) =
\begin{cases}
E(z), \ &z \in \Omega_1,
\\
E(z)
\begin{pmatrix}
1 & 0
\\
-\I & 1
\end{pmatrix}, \ &z \in \Omega_2,
\\
E(z)
\begin{pmatrix}
0 & \I
\\
\I & 1
\end{pmatrix}, \ &z \in \Omega_3,
\\
E(z)
\begin{pmatrix}
1 & \I
\\
0 & 1
\end{pmatrix}, \ &z \in \Omega_4.
\end{cases}
\eeq

Next we try to find the explicit solution. For this it is crucial that the jump matrix $v_\cA$ is piecewise constant. This implies that $\Upsilon''$ also satisfies the Airy R-H problem, except for the condition (iii), meaning that
\beq
F(z) := \Upsilon''(z) \Upsilon(z)^{-1}
\eeq
must have no jumps and therefore be an entire matrix-valued function. Looking at the leading term of $\Upsilon$ at infinity, one might wrongly conclude that $F(z)$ is equal to $z$ times the identity matrix. This would imply that all entries of $\Upsilon$ satisfy the Schr\"odinger equation with a linear potential
\beq \label{SE}
-\psi'' + z \psi = 0.
\eeq
in each sector $\Omega_i$. This will turn out to be incorrect. 

Still, we are lead to consider special solutions of \eqref{SE}, as they might lead to the right asymptotical behaviour at infinity.
One solution is the Airy function $\Ai(z)$  which satisfies \cite[Eq.~ 9.7.5]{dlmf}
\beq \label{asymAi}
\Ai(z) = \dfrac{1}{2\sqrt{\pi}}\Big(z^{-1/4}+ O(z^{-7/4})\Big) \E^{-2/3 z^{3/2}}, 
\hspace{10pt} |\arg(z)| < \pi,
\eeq
as $z \rightarrow \infty$, where all roots have branch cuts on $\R_-$. The convergence at infinity is uniform in any closed sector not containing $\R_-$. Similarly, for the first derivative of $\Ai$ \cite[Eq.~9.7.6]{dlmf} we get:
\beq \label{asym2Ai}
\Ai'(z) = -\dfrac{1}{2\sqrt{\pi}}\Big(z^{1/4}+ O(z^{-5/4}) \Big) \E^{-2/3 z^{3/2}},
\hspace{10pt} |\arg(z)| < \pi,
\eeq
with the same constraints on the convergence  rate at infinity. Note that $\Ai'$ is no longer a solution of \eqref{SE}. 

As the functions $\Ai$ and $\Ai'$ are entire, we can use them to build the entire matrix-valued function $E(z)$. In order to match the asymptotics of $\Upsilon$ in $\Omega_1$ as $z \rightarrow \infty$, we conjecture that:
\beq
E(z) = 
\begin{pmatrix}
\Ai(z) & \Ai(\xi^2 z)
\\
\Ai'(z) &  \xi^2\Ai'(\xi^2 z)
\end{pmatrix}\E^{\pi \I \sigma_3 /12 },
\eeq
with $\xi := \E^{2\pi\I/3}$. It turns out that while the asymptotics of the Airy functions hold uniformly only away from the negative real axis, the asymptotics of $\Upsilon$ given by \eqref{PE} will be uniform in all directions and have the required form. This follows from the connection formula \cite[9.2.12]{dlmf}:
\beq \label{relAi}
\Ai(z) + \xi\Ai(\xi z) + \xi^2\Ai(\xi^2 z) = 0.
\eeq
Note that $\Ai(\E^{\I \varphi} z)$ has an asymptotic expansion at infinity which holds uniformly away from the ray with angle $\pi-\varphi$.

As usual, using a Liouville type argument one can show that the solution $\Upsilon$ is unique, confirming the prior statement, that not all entries of $\Upsilon$ can be locally solutions of \eqref{SE}. In fact we can now directly compute
\beq
F(z) = 
\begin{pmatrix}
z & 0
\\
1 & z
\end{pmatrix}.
\eeq

\section{Local a priori \texorpdfstring{$L^p$}{}-estimate}
Let us consider the general setting of Section 3, with the additional assumption that $\Sigma^{exp}$ is unbounded. This was not the case in our application to orthogonal polynomials on $[-1,1]$, for further examples of R-H problems having a bounded $\Sigma^{exp}$ see (\cite{BaikDeift}, \cite{KamvissisToda}, \cite{KT}). However, there are numerous examples of R-H problems having an unbounded $\Sigma^{exp}$ (\cite{RAREFACTION}, \cite{DKMVZ2}, \cite{DKMVZ1}, \cite{DZ}, \cite{DZNLS}, \cite{DZPainleve}, \cite{DZaPriori}, \cite{EPT}, \cite{GT}, \cite{KotlyarovmKdV}, \cite{JL1}, \cite{ParametrixFinal}). 
Figure \ref{fig3} displays the contour in the case the KdV equation with steplike initial data \cite{EPT}, where the solid part is $\Sigma^{mod}$ and the dashed $\Sigma^{exp}$. We see that $\Sigma^{exp}$ extends to $\pm \infty$.

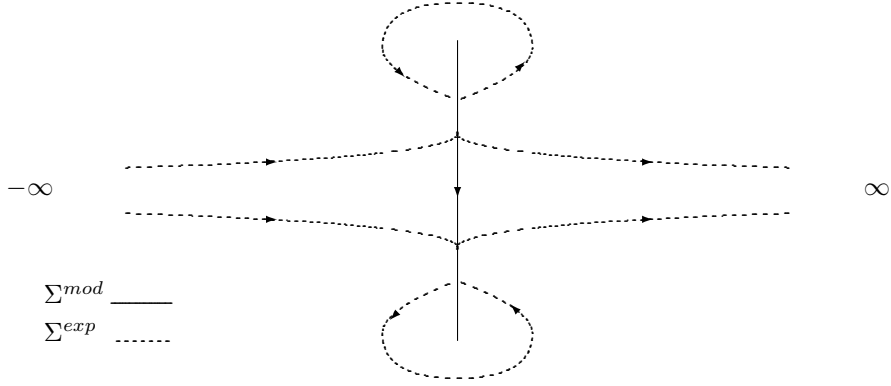
\begin{figure}[ht]
\begin{picture}(8,4.5)
\put(4,0){\line(0,1){4}}
\put(4,2){\vector(0,-1){0.1}}

\put(6.5,2.38){\vector(1,0){0.1}}
\put(6.5,1.62){\vector(1,0){0.1}}
\put(1.5,2.38){\vector(1,0){0.1}}
\put(1.5,1.62){\vector(1,0){0.1}}

\put(4.8,3.62){\vector(1,1){0.1}}
\put(3.2,3.61){\vector(1, -1){0.1}}
\put(4.8,0.39){\vector(-1,1){0.1}}
\put(3.2,0.38){\vector(-1,-1){0.1}}
\put(-2,1.95){$-\infty$}
\put(9.4,1.95){$\infty$}
\put(-1.5,0.5){$\Sigma^{mod}$} 
\curve(-0.6,0.5, 0.2,0.5)

\curvedashes{0.05,0.05}
\curve(4,3.2, 3,3.9, 4,4.5, 5,3.9, 4,3.2)
\curve(4,0.8, 3,0.1, 4,-0.5, 5,0.1, 4,0.8)
\curve(4,2.8, 5,2.5, 8.5,2.3)
\curve(4,1.2, 5,1.5, 8.5,1.7)
\curve(4,2.8, 3,2.5, -0.5,2.3)
\curve(4,1.2, 3,1.5, -0.5,1.7)

\put(-1.5,-0){$\Sigma^{exp}$} 
\curve(-0.6,-0, 0.2,-0)
\end{picture}
\vspace{10pt}
  \caption{Contour for the KdV equation with steplike initial data}\label{fig3}
\end{figure}

The unboundedness assumption on $\Sigma^{exp}$ poses a hurdle for obtaining the a priori $L^p$-estimate. Namely, given that $S(z) \rightarrow \id$, as $z \rightarrow \infty$, we see that
\beq \label{aprioriInfty}
\Vert S_- \Vert_{L^p(\Sigma^{exp})} = \infty, \quad \mbox{for} \ p \in (1,\infty).
\eeq
Luckily, it turns out that only local $L^p$-estimates around the points $\kappa \in \cK$ are needed. To show this, we choose a bounded domain $\Delta \subset \C$, which contains all the oscillatory points $\kappa \in \cK$. Next, write $\Sigma^{exp}$ as a disjoint union of an unbounded part $\Sigma^{exp}_\infty$ and a bounded part $\Sigma^{exp}_\cK$ with $\cK \subset \Sigma^{exp}_\cK$:
\beq
\Sigma^{exp}_\infty := \Sigma^{exp} \setminus \Delta, \quad \quad \Sigma^{exp}_\cK := \Sigma^{exp} \cap \Delta.
\eeq
We make the additional assumption
\beq \label{momentsW}
\Vert k^{i}w_\cS(k,n) \Vert_{L^1(\Sigma_\infty^{exp})} = O(\E^{-cn})
\eeq
for some positive $c$ and $i = 0,\dots,\ell-1$, where $\ell$ will play the same role as in Theorem \ref{mainTheorem}. The condition \eqref{momentsW} is satisfied in most applications. 

We can now introduce the following R-H problem with data $(v_{\cG}, \Sigma^\cG)$, where $\Sigma^\cG := \Sigma^{mod} \cup \Sigma^{exp}_\infty$ and the jump matrix is given by $v_\cG(k) = 
v_\cS(k)$ for $k \in \Sigma^\cG$.
\\
\\
Find a $2 \times 2$ matrix-valued function $G$ such that
\vspace{10pt}
\begin{enumerate}[(i)]
\item $G(z)$ is analytic for $z \in \C \setminus \Sigma^\cG$, 
\\
\item $G_+(k) = G_-(k)v_\cG(k)$, \ for $k \in \Sigma^\cG$
\\
\item $G(z) = \id + O(z^{-1}), \quad \mbox{as } z \rightarrow \infty
$. 
\end{enumerate}
\vspace{10pt}
The solution $G$ will be later used as a substitute for the model solution $N$. From \eqref{momentsW} it follows that
\beq \label{ws1}
\Vert v_\cG - v_\cN \Vert_{L^1(\Sigma)} = \Vert w_\cS \Vert_{L^1(\Sigma^{exp}_\infty)} =  O(\E^{-cn})
\eeq 
Moreover, in Section 3 it was assumed that $|w_\cS(k,n)| = O(\E^{-cn})$, without loss of generality with the same $c$, uniformly away from the points $\kappa \in \cK$. Hence, we also have
\beq \label{wsinfty}
\Vert v_\cG - v_\cN \Vert_{L^\infty(\Sigma)} = \Vert w_\cS \Vert_{L^\infty(\Sigma^{exp}_\infty)} =  O(\E^{-cn}).
\eeq
Let us now take a look at the corresponding singular integral equations for both R-H problems (cf. Section 2.1)
\beq \label{singG}
(\id -\cC^{\Sigma}_{w_\cG}) \Phi = \cC_-^{\Sigma}(w_\cG),
\eeq
\beq \label{singN}
(\id -\cC^{\Sigma}_{w_\cN}) \Phi = \cC_-^{\Sigma}(w_\cN),
\eeq
where $w_\cG := v_\cG - \id$, and for notational convenience we regard everything defined on the larger contour $\Sigma$. Note that
\begin{align} \label{GN}
\Vert (\id -\cC^{\Sigma}_{w_\cG})-(\id -\cC^{\Sigma}_{w_\cN}) \Vert_{L^q(\Sigma)} &= \Vert \cC^{\Sigma}_{w_\cN-w_\cG} \Vert_{L^q(\Sigma)} 
\\\nonumber
&\leq C(q) \Vert w_\cS \Vert_{L^\infty(\Sigma^{exp}_\infty)} = O(\E^{-cn}).
\end{align}
Moreover, as for $q \in (1,4)$ a unique solution $\Phi_\cN := N_--\id$ of \eqref{singN} exists, we know that $\id -\cC^{\Sigma}_{w_\cN}$ must be invertible as an operator on $L^q(\Sigma)$. As the set of invertible operators is open in the operator norm topology, it follows from \eqref{GN} that for $n$ large enough $\id -\cC^{\Sigma}_{w_\cG}$ is uniformly invertible on $L^q(\Sigma)$. Moreover, using \eqref{ws1} and \eqref{wsinfty} we see that 
\beq
\Vert w_\cG - w_\cN \Vert_{L^q(\Sigma)} = \Vert w_\cS \Vert_{L^q(\Sigma^{exp}_\infty)} =  O(\E^{-cn}).
\eeq
Altogether this implies that for $n$ large enough the unique solution $\Phi_\cG$ of \eqref{singG} satisfies
\beq
\Vert \Phi_\cG - \Phi_\cN \Vert_{L^q(\Sigma)} = O(\E^{-cn}). 
\eeq
In particular, $\Vert \Phi_\cG \Vert_{L^q(\Sigma)}$ is uniformly bounded as $\Phi_\cN$ is $n$-independent.

As the goal will be to use $G$ as a substitute of the model solution $N$, we need to understand the behaviour of $G$ in the vicinities of the oscillatory points $\kappa \in \cK$. As described in Section 2.2, we can relate $G$ and $N$ via:
\beq
G(z,n) = \Bigg( \id + \dfrac{1}{2\pi \I} \int_{\Sigma_\infty^{exp}} \dfrac{G_-(k,n)w_\cS(k,n) N^{-1}(k)}{k-z} \ dk \Bigg) N(z).
\eeq
Note that
\beq
\Vert G_-w_\cS N^{-1} \Vert_{L^1(\Sigma^{exp}_\infty)} = \Vert (\Phi_\cG + \id)w_\cS N^{-1} \Vert_{L^1(\Sigma^{exp}_\infty)}
\eeq
\beq \nonumber
\leq \Big( \Vert \Phi_\cG \Vert_{L^2 (\Sigma^{exp}_\infty)} \Vert w_\cS \Vert_{L^2 (\Sigma^{exp}_\infty)} + \Vert w_\cS \Vert_{L^1 (\Sigma^{exp}_\infty)} \Big) \Vert N^{-1} \Vert_{L^\infty (\Sigma^{exp}_\infty)} = O(\E^{-cn}),
\eeq
which follows from the uniform boundedness of $\Vert \Phi_\cG \Vert_{L^2(\Sigma)}$, and $|N^{-1}(k)|$ away from the points $\kappa \in \cK$. Hence, we conclude that
\beq \label{GNrelation}
G(z,n) = (\id + O(\E^{-cn}\dist(z,\Sigma^{exp}_\infty)^{-1}))N(z).
\eeq
In particular, we see that $G$ has locally around $\kappa \in \cK$ the same behaviour as $N$,
\beq
\label{Gkappa}
| G(z,n) | = O(| z - \kappa |^{-1/4}),
\eeq
uniformly for $n \rightarrow \infty$. However, $G(z,n)$ might not be uniformly bounded away from the points $\kappa \in \cK$, as it can blow up near the contour $\Sigma^{exp}_\infty$. 

Let us now reconsider \eqref{CS}, but with $G$ instead of $N$:
\begin{align} 
\big \Vert S_-(v_\cS-v_\cG)G_+^{-1} \big \Vert_{L^1(\Sigma)}
&=
\big\Vert S_-w_\cS G^{-1} \big \Vert_{L^1(\Sigma^{exp}_\cK)}
\notag\\
&\leq \Vert S_- \Vert_{L^p(\Sigma^{exp}_\cK)} 
\big \Vert w_\cS G^{-1} \big \Vert_{L^q(\Sigma^{exp}_\cK)}.
\end{align}
Because of $\det G \equiv 1$ and the boundedness of $\Sigma^{exp}_\cK$, we have
\beq
\Vert G^{-1} \Vert_{L^q(\Sigma^{exp}_\cK)} = \Vert G \Vert_{L^q(\Sigma^{exp}_\cK)} = \Vert \Phi_\cG + \id \Vert_{L^q(\Sigma^{exp}_\cK)} \leq C < \infty \eeq
 Together with the local assumption \eqref{vkappa} on $w_\cS$ and \eqref{Gkappa}, it follows that
\beq
\Vert k^{i}w_\cS(k,n) G^{-1}(k)  \Vert_{L^q(\Sigma^{exp}_\cK)} = O(n^{\frac{1}{4\chi}-\frac{1}{q\chi }}), \quad q \in [1,4), \ i \in \mathbb{N}_0,
\eeq
which is in line with \eqref{estimateN}. Moreover, we see that we are now required to find only a local a priori $L^p$-estimate of the form
\beq \label{localapriori}
\Vert S_- \Vert_{L^p(\Sigma^{exp}_\cK)} = O(n^r). 
\eeq
Thus, Theorem \ref{mainTheorem} remains valid with only the local a priori $L^p$-estimate for $S_-$ \eqref{localapriori} instead of \eqref{apriori} and $G$ substituted for $N$. In fact, because of \eqref{GNrelation} the two relations \eqref{Sasymptotics1} and \eqref{Sasymptotics2} between $S$ and $N$ continue to be true and Theorem \ref{mainTheorem} holds in its original form with condition \eqref{localapriori} instead of \eqref{apriori}.
\\
\\
\noindent{\bf Acknowledgments.} The author thanks Andrei Mart\'inez-Finkelshtein for valueable discussions on this topic and related literature and Gerald Teschl for useful comments on the manuscript.

\end{document}